\documentclass[11pt]{article}
\usepackage[T1]{fontenc}
\usepackage[utf8]{inputenc}
\usepackage{lmodern}
\usepackage{color}
\usepackage{amsmath,amssymb,amsthm}
\usepackage[margin=1in]{geometry}
\usepackage[hidelinks]{hyperref}
\numberwithin{equation}{section}
\newtheorem{theorem}{Theorem}[section]
\newtheorem{proposition}[theorem]{Proposition}
\newtheorem{lemma}[theorem]{Lemma}

\theoremstyle{remark}

\newcommand{\R}{\mathbb R}
\newcommand{\F}{\mathcal F}
\newcommand{\Oc}{\mathcal O_c}
\newcommand{\ip}[2]{\langle #1,#2\rangle}
\newcommand{\norm}[1]{\lVert #1\rVert}
\newcommand{\ess}{\mathrm{ess}}
\DeclareMathOperator{\Span}{span}
\DeclareMathOperator{\dist}{dist}
\title{Nonlinear coercivity and orbital stability of smooth solitary waves for the modified Camassa--Holm equation}
\author{Xijun Deng\textsuperscript{1},\quad Ji Li\textsuperscript{2}\thanks{Corresponding author. }\,,  \quad Zhisu Liu\textsuperscript{3}\\[6pt]
\small\textsuperscript{1}Department of Mathematics, Hubei University of Automotive Technology,\\
\small Shiyan, Hubei 442002, P. R. China\\
\small\texttt{xijundeng@yeah.net}\\
\small\textsuperscript{2} School of Mathematics and Statistics,\\
\small Huazhong University of Science and Technology, Wuhan, Hubei, 430074, P. R. China \\
\small\texttt{liji@hust.edu.cn}\\
\small\textsuperscript{3} School of Mathematics and Physics,\\
\small China University of Geosciences, Wuhan, Hubei, 430074, P. R. China \\
\small\texttt{liuzhisu@cug.edu.cn}}
\date{}
\begin{document}
\maketitle
\begin{abstract}
In this paper, we establish orbital stability of smooth zero-background solitary waves for the modified Camassa--Holm
equation in $H^1(\R)\cap W^{1,4}(\R)$ for the full range $\gamma<c<2\gamma$. The main difficulty is that spectral coercivity
of the Hessian does not by itself yield nonlinear coercivity in this topology, because normalized perturbations may develop
derivative concentration. In order to overcome this obstruction, we apply a penalized
Lyapunov functional and introduce a derivative factorization.
The key ingredient is a concentration-coercivity principle for the normalized Lyapunov energy, which yields local nonlinear
Lyapunov coercivity. Our result extends the $H^1(\mathbb R)\cap W^{1,4}(\mathbb R)$ orbital-stability results of Li and Liu
[Ann. PDE, 2021] and Zhang and Zhu [JDE, 2025] to the entire parameter range $\gamma<c<2\gamma$.
\end{abstract}

\section{Introduction}\label{sec:intro}
The classical orbital-stability theory of Benjamin \cite{ref1}, Weinstein \cite{ref29} and Grillakis, Shatah and Strauss \cite{ref11,ref12} relates the stability of a solitary wave to the spectral structure of the Hessian of a conserved action and a slope condition along the solitary-wave branch. Once the symmetry and constraint directions have been taken into account, these conditions provide coercivity of the associated quadratic form. For a Lyapunov argument, however, this coercivity must also hold at the nonlinear level in the topology under consideration. This passage is not automatic and can be rather challenging when the conserved action contains higher powers of derivatives. The modified Camassa--Holm (mCH) equation with linear dispersion serves as a typical example.

The equation reads
\begin{equation}\label{eq:1.1}
m_t+\bigl((u^2-u_x^2)m\bigr)_x+\gamma u_x=0,\qquad m=u-u_{xx},\qquad \gamma>0,
\end{equation}
for a real-valued function $u=u(t,x)$ on $\R_+\times\R$. The equation belongs to the class of integrable shallow-water models 
associated with the Camassa–Holm equation \cite{ref2} and the tri-Hamiltonian duality developed in \cite{ref9,ref10,ref23,ref24}.
 Other integrable Camassa–Holm-type equations with cubic nonlinearities have been investigated in a broader 
 context \cite{ref14,ref15}. For the mCH equation, its qualitative theory includes wave breaking, peakons and smooth 
 solitary waves; see, among others, \cite{ref3,ref5,ref8,ref13,ref21,ref22,ref25,ref26}. Of particular interest here are the smooth zero-background
solitary waves $\phi_c(x-ct)$, which exist, up to translation and sign, for the wave-speed regime $\gamma<c< 2\gamma$;
see~\cite{ref17,ref31}. The positive profile $\phi_c$ is even, strictly decreasing on $(0,\infty)$,
and exponentially decaying.

The natural energy topology is
\[X:=H^1(\R)\cap W^{1,4}(\R),\qquad \norm{u}_X:=\norm{u}_{H^1}+\norm{u}_{W^{1,4}}.\]
The two conserved quantities are
\begin{equation}\label{eq:2.3}
E(u)=\frac12\int_\R(u^2+u_x^2)\,dx
\end{equation}
and
\begin{equation}\label{eq:2.4}
H(u)=\frac14\int_\R\left(-u^4-2u^2u_x^2+\frac13u_x^4-2\gamma u^2\right)\,dx.
\end{equation}
Set
\begin{equation}\label{eq:2.5}
Q_c(u):=H(u)+cE(u),
\end{equation}
then $Q_c'(\phi_c)=0$.
Here the conserved Hamiltonian contains a quartic derivative term $u_x^4$ and the indefinite mixed term $u^2u_x^2$.
Consequently, \emph{the main obstacle is that
a quadratic expansion at a solitary wave does not by itself provide the nonlinear control needed in $X$.}
Establishing orbital
 stability in $X$ is therefore substantially different from proving stability in a stronger Sobolev space:
 perturbations may be small in $X$ while possessing large higher derivatives, and the nonlinear derivative terms
 cannot be treated solely through higher-regularity estimates.

The orbital stability of smooth solitary waves for the mCH equation has been investigated in several recent works.
 Li and Liu \cite{ref17} and Zhang and Zhu \cite{ref31} studied orbital stability of smooth solitary waves of
 \eqref{eq:1.1} by global constrained minimization. Its main bottleneck is the recovery of strong $H^1\cap W^{1,4}$
 compactness of minimizing sequences by employing the compactness-concentration principle established in \cite{ref19,ref20}: additional inequalities between the conserved quantities are used to control
 the mixed term $u^2u_x^2$, and these inequalities restrict the accessible wave-speed regime.
 In particular, Zhang and Zhu \cite{ref31} proved orbital stability for $\gamma<c<\hat c$, where $\hat c\simeq1.27\gamma$, extending the original work of Li and Liu \cite{ref17}. Li, Liu and Zhu \cite{ref18} obtained stability throughout
 the full open smooth-wave range in the stronger $H^3$ topology, using conserved quantities in the momentum variable.
 Deng, Lafortune and Liu \cite{ref6} treated a distinct family of smooth waves on nonzero background in a momentum-space framework.
Consequently, orbital stability of zero-background smooth solitary waves in the natural energy
topology $X$, without the additional regularity required by the $H^3$ theory, remains unresolved for $c\in[\hat c,2\gamma)$. We solve this problem.

The analytical core of the paper is mainly the following three steps:\\
\emph{Step 1}. An observation from \cite{ref17} that
the Hessian operator
$$
L_c=Q_c''(\phi_c)=-\partial_x(A_c\partial_x)+A_c-2\phi_c(\phi_c-\phi_{c,xx})-\gamma
$$
of the conserved quantity, with uniform coefficient bound
\begin{equation}\label{eq:1.3}
A_c(x):=c-\phi_c^2+\phi_{c,x}^2\geq 2\gamma-c>0,\qquad \gamma<c<2\gamma,
\end{equation}
enables a detailed analysis of the spectrum of operator $L_c$. Particularly, $L_c$ has exactly one negative eigenvalue, $\ker L_c=\Span\{\phi_{c,x}\}$, and
\[\frac{d}{dc}E(\phi_c)=\sqrt{\frac{c}{c-\gamma}}>0.\]
These properties lead to constrained spectral coercivity.

\emph{Step 2}. We state an exact nonnegative derivative-energy factorization.
Expanding $Q_c$ around $\phi_c$, we separate the derivative energy
\[
B_c(x,v,w)
=\frac12(A_c-2\phi_cv-v^2)w^2
+\frac13\phi_{c,x}w^3
+\frac1{12}w^4.
\]
At $v=0$, this expression admits the exact factorization (see Proposition \ref{prop:factorization})
\begin{equation}\label{eq:1.3+}
B_c(x,0,w)
=\omega_c(x)w^2
+\frac1{12}w^2(w+2\phi_{c,x})^2,
\end{equation}
where
\begin{equation}\label{eq:1.4}
\omega_c(x)
=\frac12(c-\phi_c^2)+\frac16\phi_{c,x}^2
\geq\frac{2\gamma-c}{2}>0.
\end{equation}
This identity (\ref{eq:1.3+}) provides more information than the nonnegativity of the derivative energy, whose two terms play complementary roles in the nonlinear coercivity argument. The strictly positive
weighted quadratic term assigns a definite energy cost to derivative concentration that is not retained
by the weak limit. Meanwhile, the nonnegative square term recovers the additional derivative contribution needed to
reconstruct the complete Hessian quadratic form on the weak component. Thus the factorization simultaneously preserves
the correct linearized energy and controls the nonlinear defect associated with weak convergence.


\emph{Step 3}. We introduce a penalized Lyapunov
functional defined by
\[
\F_c(u):=Q_c(u)-Q_c(\phi_c)+K_c\bigl(E(u)-E(\phi_c)\bigr)^2
\]
to establish local nonlinear Lyapunov coercivity without relying on global compactness of constrained minimizing sequences.

More precisely, we establish \emph{a concentration-coercivity principle} to $\F_c$.
For normalized perturbations $z_n=\frac{v_n}{||v_n||_{H^1}}$, weak $H^1$-convergence may
leave a derivative defect measure $d\mu^*$.
The resulting lower bound is that if $\ip{z_n}{\phi_x}=0$, then
\begin{equation}\label{eq:1.6}
\liminf_{n\to\infty}\frac{\F_c(\phi_c+v_n)}{r_n^2}
\geq\mu_c\norm{z}_{H^1}^2+\frac{2\gamma-c}{2}d\mu^*(\R),\qquad\mu_c>0.
\end{equation}
Thus both the weak component and the local derivative defect carry positive energy. Far-field positivity excludes escape of normalized mass to infinity, yielding local nonlinear Lyapunov coercivity; conservation of the Hamiltonian then recovers the $W^{1,4}$-control required for orbital stability.


Our main result is stated as follows.
\begin{theorem}\label{thm:stability}
Let $\gamma>0$ and let $\phi_c(x-ct)$ be the positive smooth zero-background solitary wave of \eqref{eq:1.1} with $\gamma<c<2\gamma$.
Then $\phi_c$ is orbitally stable in $X$ up to the maximal existence time of the corresponding strong solution. More precisely, for every $\varepsilon>0$ there exists $\delta>0$ such that, if $u_0\in H^s(\R)$ with $s>\frac{5}{2}$ and
\[\norm{u_0-\phi_c}_X<\delta,\]
then the associated solution $u(t)$ satisfies
\begin{equation}\label{eq:stability}
\sup_{0\leq t<T}\inf_{y\in\R}\norm{u(t)-\phi_c(\cdot-y)}_X<\varepsilon.
\end{equation}
The same conclusion holds for the negative solitary wave $-\phi_c$.
\end{theorem}
Here $T\in(0,\infty]$ is the maximal existence time of the strong solution.
The theorem controls the orbit throughout this lifespan; it does not assert global existence
for arbitrary nearby data.

We remark that the use of penalized (or augmented) Lyapunov functional is classical, see \cite{MS93,Stuart08}.
The contribution here is a nonlinear coercivity mechanism adapted to the topology $H^1\cap W^{1,4}$: an exact factorization of the derivative contribution yields a lower bound that controls the weak $H^1$ limit and its derivative defect measure separately. A further far-field argument excludes loss of normalized mass at infinity and upgrades the spectral coercivity to local nonlinear Lyapunov coercivity.
The argument in this paper illustrates a general methodological possibility:
higher-order derivative nonlinearities that obstruct a standard quadratic Lyapunov expansion may possess exact positive structures
capable of controlling the defects responsible for that obstruction. The concentration-coercivity mechanism developed here suggests
a useful direction for studying other variational stability problems in which derivative concentration prevents
a direct passage from spectral to nonlinear coercivity.

Sections 2 and 3 contain results either are already known or follow from standard arguments, and are included for completeness and readability:
Section~\ref{sec:variational} collects the variational framework and solitary-wave identities. Section~\ref{sec:spectral} establishes the spectral structure and differentiability of the solitary-wave branch, then uses the slope condition to obtain constrained coercivity and penalized spectral coercivity. Section~\ref{sec:nonlinear} contains the main novelty. It proceeds from exact nonlinear derivative factorization to a concentration-coercivity principle separating the weak $H^1$ component and the derivative defect, and then uses far-field positivity to obtain nonlinear Lyapunov coercivity. Finally, Section~\ref{sec:stability} completes the orbital-stability argument by quantitative modulation and conservation laws.

\section{Well-posedness and solitary-wave identities}\label{sec:variational}
We first recall the well-posedness and conservation properties needed below. The Cauchy theory for the integrable 
mCH equation was developed in Besov and Sobolev spaces by Fu, Gui, Liu and Qu \cite{ref8}; 
see also \cite{ref7,ref13,ref30}. In particular, when $p=r=2$ their local theory 
yields the Sobolev threshold $s>\frac{5}{2}$.

\begin{proposition}\label{prop:wellposed}
Let $s>\frac{5}{2}$ and $u_0\in H^s(\R)$. Then there exists a maximal time $T\in(0,\infty]$ and a unique strong solution of \eqref{eq:1.1} satisfying
\begin{equation}\label{eq:2.1}
u\in C([0,T);H^s(\R))\cap C^1([0,T);H^{s-1}(\R)).
\end{equation}
The data-to-solution map is continuous in this class. Along such solutions the functionals $E$ and $H$ defined above are conserved.
\end{proposition}
The local well-posedness assertion follows from \cite[Theorem 3.1 and Remark 3.1]{ref8}. Conservation for strong solutions follows from the Hamiltonian identities, and is the conservation framework used in \cite{ref17,ref18,ref31}. Since $s>\frac{5}{2}$, Sobolev embedding gives $H^s(\R)\hookrightarrow W^{1,\infty}(\R)$; hence \eqref{eq:2.1} implies
\begin{equation}\label{eq:2.2}
u\in C([0,T);X).
\end{equation}

A direct Euler--Lagrange calculation gives
\begin{equation}\label{eq:2.6}
Q_c'(u)=(c-u^2+u_x^2)(u-u_{xx})-\gamma u.
\end{equation}
Thus a smooth solitary-wave profile $\phi=\phi_c$ satisfies
\begin{equation}\label{eq:2.7}
(c-\phi^2+\phi_x^2)(\phi-\phi_{xx})-\gamma\phi=0.
\end{equation}
The phase-plane analysis in \cite{ref17,ref31} gives the first integral
\begin{equation}\label{eq:2.8}
\phi_x^2=\phi^2-c+\sqrt{c^2-2\gamma\phi^2}.
\end{equation}
The positive profile is even, strictly decreasing for $x>0$, and decays exponentially together with its derivatives; these properties follow from the phase-plane description and the linearization at the zero equilibrium, and are standard in the solitary-wave construction in \cite{ref17,ref22,ref31}.
Define
\begin{equation}\label{eq:2.9}
A_c(x):=c-\phi^2+\phi_x^2.\end{equation}
By \eqref{eq:2.8},
\begin{equation}\label{eq:2.10}
A_c=\sqrt{c^2-2\gamma\phi^2}.
\end{equation}
At the crest $x=0$, $\phi_x(0)=0$, and \eqref{eq:2.8} yields
\begin{equation}\label{eq:2.11}
\phi^2(0)=2(c-\gamma).
\end{equation}
Since $\phi(0)$ is the maximum of the positive profile,
\begin{equation}\label{eq:2.12}
A_c(x)\geq2\gamma-c>0
\end{equation}
for every $\gamma<c<2\gamma$.
The conserved momentum evaluated on the solitary wave is
\begin{equation}\label{eq:2.13}
E(\phi_c)=\sqrt{c(c-\gamma)}+\frac\gamma2\log\frac{(\sqrt{c-\gamma}+\sqrt c)^2}{\gamma}.
\end{equation}
Differentiating $E(\phi_c)$ gives the useful slope identity
\begin{equation}\label{eq:2.14}
\frac{d}{dc}E(\phi_c)=\sqrt{\frac{c}{c-\gamma}}>0.
\end{equation}
The two identities that will drive the stability analysis are therefore the uniform positivity of $A_c$, which controls the principal part of the Hessian, and the positive momentum slope. Their nonlinear significance will emerge in the exact factorization and the concentration-coercivity principle of Section~\ref{sec:nonlinear}.

\section{Spectral analysis and linear coercivity}\label{sec:spectral}
In this section, we establish the spectral properties of $L_c$ and the coercivity estimates needed for the nonlinear argument. Throughout this section, $\gamma<c<2\gamma$.
\subsection{The linearized operator and its spectrum}
Linearizing \eqref{eq:2.6} gives
\begin{equation}\label{eq:3.1}
L_cv=A_c(v-v_{xx})+(-2\phi v+2\phi_xv_x)(\phi-\phi_{xx})-\gamma v.
\end{equation}
Since
\begin{equation}\label{eq:3.2}
A_c'=-2\phi_x(\phi-\phi_{xx}),
\end{equation}
we can write
\begin{equation}\label{eq:3.3}
L_c=-\partial_x(A_c\partial_x)+V_c,
\end{equation}
where
\begin{equation}\label{eq:3.4}
V_c=A_c-2\phi(\phi-\phi_{xx})-\gamma.
\end{equation}
From \eqref{eq:2.7} and \eqref{eq:2.10}, we obtain
\begin{equation}\label{eq:3.5}
\phi-\phi_{xx}=\frac{\gamma\phi}{A_c},\qquad\phi^2=\frac{c^2-A_c^2}{2\gamma},
\end{equation}
and therefore
\begin{equation}\label{eq:3.6}
V_c=2A_c-\gamma-\frac{c^2}{A_c}.
\end{equation}
In particular,
\begin{equation}\label{eq:3.7}
A_c(x)\to c,\qquad V_c(x)\to c-\gamma\quad(|x|\to\infty).
\end{equation}

\begin{lemma}\label{lem:selfadjoint}
For every fixed $\gamma<c<2\gamma$, the quadratic form
\begin{equation}\label{eq:3.8}
q_c[v]=\int_\R(A_c|v_x|^2+V_c|v|^2)\,dx,\qquad D(q_c)=H^1(\R),
\end{equation}
is closed and bounded from below. It determines a unique self-adjoint operator, again denoted by $L_c$, with
\begin{equation}\label{eq:3.9}
D(L_c)=H^2(\R).
\end{equation}
Moreover,
\begin{equation}\label{eq:3.10}
\sigma_{\ess}(L_c)=[c-\gamma,\infty).
\end{equation}
\end{lemma}

\begin{proof}
By \eqref{eq:2.12}, $A_c\geq a_c:=2\gamma-c>0$. The smoothness and exponential decay of $\phi$ imply
\[A_c,A_c',V_c\in L^\infty(\R),\qquad A_c^{-1}\in L^\infty(\R).\]
Hence $q_c[v]\geq a_c\norm{v_x}_2^2-\norm{V_c^-}_\infty\norm{v}_2^2$, and 
after adding a sufficiently large multiple of $\norm{v}_2^2$ the form norm is equivalent to the $H^1$ norm. 
Thus the form is closed and semibounded, and the representation theorem yields a unique self-adjoint operator; 
see, for example, Kato \cite{ref16}.

If $v\in D(L_c)$, then $-\partial_x(A_cv_x)+V_cv=f$ for some $f\in L^2$. Expanding the derivative gives $-A_cv_{xx}=f+A_c'v_x-V_cv$. The right-hand side belongs to $L^2$, and $A_c^{-1}\in L^\infty$, so $v_{xx}\in L^2$. Hence $v\in H^2$. The reverse inclusion follows directly from \eqref{eq:3.3}; this proves \eqref{eq:3.9}.

Finally, by exponential decay,
\[A_c-c\to0,\qquad A_c'\to0,\qquad V_c-(c-\gamma)\to0\quad(|x|\to\infty).\]
We determine the essential spectrum without treating the variable leading coefficient as a compact operator perturbation. By Persson's theorem \cite{ref27,ref28} and \eqref{eq:3.8},
\[\inf\sigma_{\ess}(L_c)=\lim_{R\to\infty}\inf_{\substack{v\in C_c^\infty(\R\setminus[-R,R])\\\norm{v}_2=1}}q_c[v].\]
Given $\varepsilon>0$, for $R$ sufficiently large one has $A_c(x)\geq c-\varepsilon$ and $V_c(x)\geq c-\gamma-\varepsilon$ for $|x|\geq R$, and therefore
\[q_c[v]\geq(c-\varepsilon)\norm{v_x}_2^2+(c-\gamma-\varepsilon)\norm{v}_2^2\]
for all such $v$. Hence $\inf\sigma_{\ess}(L_c)\geq c-\gamma$.
Conversely, fix $\lambda\geq c-\gamma$ and set $k=\sqrt{\frac{\lambda-c+\gamma}{c}}$. Choose $\chi\in C_c^\infty((-1,1))$ with $\norm{\chi}_2=1$, and let
\[\psi_n(x)=\ell_n^{-1/2}\chi\left(\frac{x-x_n}{\ell_n}\right)e^{ikx},\qquad x_n\to+\infty,\quad\ell_n\to\infty,\quad\ell_n=o(x_n).\]
Then $\norm{\psi_n}_2=1$ and $\psi_n\rightharpoonup0$ in $L^2$. On the support of $\psi_n$ the coefficients $A_c,A_c'$ and $V_c$ converge uniformly to $c,0$ and $c-\gamma$, respectively. A direct differentiation therefore gives $\norm{(L_c-\lambda)\psi_n}_2\to0$. The Weyl criterion \cite{ref27,ref28} yields $\lambda\in\sigma_{\ess}(L_c)$. Since $\lambda\geq c-\gamma$ was arbitrary, \eqref{eq:3.10} follows.
\end{proof}
Translation invariance of functional $Q_c$ yields
\begin{equation}\label{eq:3.11}
L_c\phi_x=0.
\end{equation}

\begin{proposition}\label{prop:spectrum}
For every $\gamma<c<2\gamma$, the operator $L_c$ has exactly one simple negative eigenvalue and
\begin{equation}\label{eq:3.12}
\ker L_c=\Span\{\phi_x\}.
\end{equation}
\end{proposition}

\begin{proof}
By Lemma~\ref{lem:selfadjoint}, $0<c-\gamma=\inf\sigma_{\ess}(L_c)$, so the zero eigenvalue is discrete. The eigenfunction $\phi_x$ has exactly one zero because $\phi$ is even and strictly decreasing on $(0,\infty)$. The Sturm oscillation theorem for one-dimensional uniformly elliptic Sturm--Liouville operators therefore identifies $0$ as the second discrete eigenvalue counted from below. Hence there is exactly one negative eigenvalue. Simplicity of one-dimensional eigenvalues gives \eqref{eq:3.12}; see \cite[Chs.~5--6]{ref28}.
\end{proof}

We next establish differentiation of the solitary-wave branch on wave speed $c$. Let
\[H_{even}^2=\{u\in H^2(\R):u(-x)=u(x)\},\qquad L_{even}^2=\{u\in L^2(\R):u(-x)=u(x)\}.\]

\begin{lemma}\label{lem:branch}
For every $c_0\in(\gamma,2\gamma)$ there exists a neighborhood $I$ of $c_0$ such that
\begin{equation}\label{eq:3.13}
c\mapsto\phi_c\in C^1(I;H_{even}^2).
\end{equation}
Moreover,
\begin{equation}\label{eq:3.14}
L_c\partial_c\phi_c+E'(\phi_c)=0.
\end{equation}
\end{lemma}

\begin{proof}
Define
\begin{equation}\label{eq:3.15}
G(c,u)=(c-u^2+u_x^2)(u-u_{xx})-\gamma u.
\end{equation}
Since $H^2(\R)\hookrightarrow W^{1,\infty}(\R)$, the map $G:(\gamma,2\gamma)\times H_{even}^2\to L_{even}^2$ is $C^1$. Also $G(c,\phi_c)=0$ and $D_uG(c,\phi_c)=L_c$.
By Proposition~\ref{prop:spectrum}, the full kernel is spanned by the odd function $\phi_x$, hence the restriction of $L_c$ to the even subspace has trivial kernel. Since $0$ is isolated, it belongs to the resolvent set of the restricted self-adjoint operator on $L_{even}^2$. Lemma~\ref{lem:selfadjoint} identifies the operator domain with $H^2$, and the closed graph theorem gives
\[
\norm{L_c^{-1}f}_{H^2}\leq C\norm{f}_2,\qquad f\in L_{even}^2.
\]
Thus $L_c:H_{even}^2\to L_{even}^2$ is a Banach-space isomorphism. The implicit-function theorem gives a local $C^1$ branch of even solutions. Uniqueness of the positive even decaying solitary wave identifies this branch with $c\mapsto\phi_c$. Differentiating $Q_c'(\phi_c)=0$ yields \eqref{eq:3.14}.
\end{proof}

Set
\begin{equation}\label{eq:3.16}
p_c:=E'(\phi_c)=\phi_c-\phi_{c,xx}.
\end{equation}
Since $p_c$ is even and $\phi_x$ is odd, $p_c\perp\phi_x$. Lemma~\ref{lem:branch} yields $L_c^{-1}p_c=-\partial_c\phi_c$.
Together with \eqref{eq:2.14}, the following holds
\begin{equation}\label{eq:3.17}
\ip{L_c^{-1}p_c}{p_c}=-\frac{d}{dc}E(\phi_c)=-\sqrt{\frac{c}{c-\gamma}}<0.
\end{equation}

\subsection{Constrained and penalized coercivity}
Write $q_c[v]=\ip{L_cv}{v}$ when $v\in H^1$ in the form sense.
\begin{proposition}\label{prop:constrained}
There exists $\kappa_c>0$ such that
\begin{equation}\label{eq:3.18}
q_c[v]\geq\kappa_c\norm{v}_{H^1}^2
\end{equation}
for every $v\in H^1(\R)$ satisfying
\begin{equation}\label{eq:3.19}
\ip{v}{\phi_x}=0,\qquad\ip{v}{p_c}=0.
\end{equation}
\end{proposition}

\begin{proof}
Let $-\lambda_c<0$ be the unique negative eigenvalue, with normalized eigenfunction $\chi_c$. Then
$$
H^1(\R)=\Span\{\phi_x\}\oplus\Span\{\chi_c\}\oplus\mathcal H_+,
$$
which yields $\Span\{\phi_x\}^\perp=\Span\{\chi_c\}\oplus\mathcal H_+$. So $q_c$ is positive on $\mathcal H_+$. Indeed, since $0$ is an isolated eigenvalue and the essential spectrum starts at $c-\gamma>0$,
\begin{equation}\label{eq:3.20}
\dist\bigl(0,\sigma(L_c)\setminus\{0\}\bigr)>0
\end{equation}
on the positive spectral subspace. Hence there exists $\delta_c>0$ such that
\begin{equation}\label{eq:3.21}
q_c[h]\geq\delta_c\norm{h}_2^2,\qquad h\in\mathcal H_+.
\end{equation}
Let $w_c=L_c^{-1}p_c$ and write $w_c=a\chi_c+w_+$, $w_+\in\mathcal H_+$. By \eqref{eq:3.17}, $q_c[w_c]=-\lambda_ca^2+q_c[w_+]<0$. Thus $a\ne0$ and
\begin{equation}\label{eq:3.22}
\theta_c:=\frac{q_c[w_+]}{\lambda_ca^2}\in[0,1).
\end{equation}
For $v$ satisfying \eqref{eq:3.19}, write $v=b\chi_c+v_+$, $v_+\in\mathcal H_+$. Since $p_c=L_cw_c$,
\[0=\ip{p_c}{v}=-\lambda_cab+\ip{L_cw_+}{v_+}.\]
Cauchy--Schwarz in the positive form inner product gives $\lambda_c^2a^2b^2\leq q_c[w_+]q_c[v_+]$, and therefore
\begin{equation}\label{eq:3.23}
\lambda_cb^2\leq\theta_cq_c[v_+].
\end{equation}
Hence $q_c[v]=-\lambda_cb^2+q_c[v_+]\geq(1-\theta_c)q_c[v_+]$. It follows from \eqref{eq:3.21} and \eqref{eq:3.23} that
\begin{equation}\label{eq:3.24}
q_c[v]\geq\eta_c\norm{v}_2^2
\end{equation}
for some $\eta_c>0$. On the other hand, uniform ellipticity and boundedness of $V_c$ give
\begin{equation}\label{eq:3.25}
q_c[v]\geq(2\gamma-c)\norm{v_x}_2^2-C_c\norm{v}_2^2.
\end{equation}
A positive linear combination of \eqref{eq:3.24} and \eqref{eq:3.25} yields \eqref{eq:3.18}.
\end{proof}

\begin{proposition}\label{prop:penalized}
There exist $K_c>0$ and $\mu_c>0$ such that
\begin{equation}\label{eq:3.26}
\frac12q_c[v]+K_c\ip{p_c}{v}^2\geq\mu_c\norm{v}_{H^1}^2
\end{equation}
for every $v\in H^1(\R)$ satisfying $\ip{v}{\phi_x}=0$.
\end{proposition}

\begin{proof}
Set $D_c:=\ip{p_c}{L_c^{-1}p_c}=-\sqrt{\frac{c}{c-\gamma}}<0$, $w_c=L_c^{-1}p_c$. For $v\perp\phi_x$, define $\alpha=\frac{\ip{p_c}{v}}{D_c}$ and $v_0=v-\alpha w_c$. Then $v_0\perp\phi_x$ and $\ip{p_c}{v_0}=0$. Since $w_c\in D(L_c)=H^2$ and $v_0\in H^1$, polarization of the closed form gives
\[q_c(v_0,w_c)=\ip{v_0}{L_cw_c}=\ip{v_0}{p_c}=0.\]
Thus the quadratic-form cross term in $v=v_0+\alpha w_c$ vanishes, and
\[\frac12q_c[v]+K_c\ip{p_c}{v}^2=\frac12q_c[v_0]+\alpha^2\left(\frac12D_c+K_cD_c^2\right).\]
Choose
\begin{equation}\label{eq:3.27}
K_c>-\frac1{2D_c}=\frac12\sqrt{\frac{c-\gamma}{c}}.
\end{equation}
Proposition~\ref{prop:constrained} controls $v_0$, and the decomposition $v=v_0+\alpha w_c$ then gives \eqref{eq:3.26}.
\end{proof}

\section{Nonlinear coercivity and derivative concentration}\label{sec:nonlinear}
Small perturbations in $X$ need not admit a uniform quadratic expansion of the action after normalization in $H^1$.
In this section, we now quantify the possible derivative concentration and use the resulting estimate to obtain nonlinear Lyapunov coercivity.
\subsection{Exact action expansion and derivative-energy factorization}
The action $Q_c=H+cE$ has the density
\begin{equation}\label{eq:4.1}
Q_c(u)=\int_\R\left[\frac1{12}u_x^4+\frac12(c-u^2)u_x^2+\frac12(c-\gamma)u^2-\frac14u^4\right]\,dx.
\end{equation}
Since $Q_c'(\phi)=0$, expansion at $\phi$ and integration by parts give the exact identity
\begin{equation}\label{eq:4.2}
Q_c(\phi+v)-Q_c(\phi)=\int_\R B_c(x,v,v_x)\,dx+\frac12\int_\R V_cv^2\,dx+R_c(v),\quad v\in X,
\end{equation}
where
\begin{equation}\label{eq:4.3}
B_c(x,v,w)=\frac12(A_c-2\phi v-v^2)w^2+\frac13\phi_xw^3+\frac1{12}w^4
\end{equation}
and
\begin{equation}\label{eq:4.4}
R_c(v)=\int_\R\left[\left(\frac13\phi_{xx}-\phi\right)v^3-\frac14v^4\right]\,dx.
\end{equation}
Indeed, the quadratic mixed term $-2\phi\phi_xvv_x$ contributes $(\phi\phi_x)_xv^2$ after integration by parts, producing the potential $\frac{V_c}{2}$. The remaining mixed cubic term is $-\phi_xv^2v_x$, whose integral equals $\frac13\int\phi_{xx}v^3\,dx$. The identities can first be checked for smooth compactly supported perturbations and then extended by density in $X$. Boundedness of the profile coefficients and $H^1(\R)\hookrightarrow L^\infty(\R)$ yield
\begin{equation}\label{eq:4.5}
|R_c(v)|\leq C\left(\norm{v}_{H^1}^3+\norm{v}_{H^1}^4\right).
\end{equation}

\begin{proposition}\label{prop:factorization}
Define
\begin{equation}\label{eq:4.6}
\omega_c=\frac12A_c-\frac13\phi_x^2.
\end{equation}
Then
\begin{equation}\label{eq:4.7}
\omega_c(x)=\frac12(c-\phi^2(x))+\frac16\phi_x(x)^2\geq \frac{2\gamma-c}{2}>0.
\end{equation}
For every $x,w\in\R$, one has
\begin{equation}\label{eq:4.8}
B_c(x,0,w)=\omega_c(x)w^2+\frac1{12}w^2(w+2\phi_x(x))^2.
\end{equation}
\end{proposition}

\begin{proof}
The formula for $A_c=c-\phi^2+\phi_x^2$ gives the equality in \eqref{eq:4.7}. The profile bound $\phi^2\leq2(c-\gamma)$ established in Section~\ref{sec:variational} gives its lower bound. Expanding the square in \eqref{eq:4.8} gives
\[\left(\omega_c+\frac13\phi_x^2\right)w^2+\frac13\phi_xw^3+\frac1{12}w^4=\frac12A_cw^2+\frac13\phi_xw^3+\frac1{12}w^4,\]
which is $B_c(x,0,w)$ defined by \eqref{eq:4.3}.
\end{proof}
The two terms in \eqref{eq:4.8} play distinct roles. The first term assigns a positive cost to derivative defects, whereas the second term contributes the additional coefficient $\frac{\phi_x^2}{3}$ on the weak limit and thereby reconstructs the derivative part of the Hessian.

\subsection{A concentration-coercivity principle}
Define the translation-invariant Lyapunov functional
\begin{equation}\label{eq:4.9}
\F_c(u)=Q_c(u)-Q_c(\phi)+K_c\bigl(E(u)-E(\phi)\bigr)^2.
\end{equation}
This functional is conserved along the strong solutions. The next result separates the weak $H^1$-component from the derivative defect. It does not assume tightness of the normalized $H^1$ mass.

\begin{theorem}[Concentration-coercivity principle]\label{thm:concentration}
Let $v_n\to0$ strongly in $X$, and set $r_n=\norm{v_n}_{H^1}>0$, $z_n=\frac{v_n}{\norm{v_n}_{H^1}}$. Suppose that, along a subsequence,
\[z_n\rightharpoonup z\text{ in }H^1(\R),\qquad |z_{n,x}|^2\,dx\rightharpoonup^*d\mu\text{ locally as Radon measures on }\R,\]
where $d\mu$ initially arises as a locally finite nonnegative Radon measure. The uniform $L^2$ bound on $z_{n,x}$ makes it finite globally. Moreover, $d\mu\geq|z_x|^2\,dx$, and
\[d\mu^*:=d\mu-|z_x|^2\,dx\geq0,\qquad d\mu^*(\R):=\sup_{R>0}d\mu^*([-R,R])<\infty.\]
One has
\begin{equation}\label{eq:4.10}
\liminf_{n\to\infty}\frac1{\norm{v_n}_{H^1}^2}\int_\R B_c(x,v_n,v_{n,x})\,dx\geq\frac12\int_\R A_c|z_x|^2\,dx+\int_\R\omega_c\,d\mu^*
\end{equation}
and
\begin{equation}\label{eq:4.11}
\liminf_{n\to\infty}\frac{\F_c(\phi+v_n)}{\norm{v_n}_{H^1}^2}\geq\frac12q_c[z]+K_c\ip{p_c}{z}^2+\int_\R\omega_c\,d\mu^*.
\end{equation}
If additionally $\ip{z_n}{\phi_x}=0$, then
\begin{equation}\label{eq:4.12}
\liminf_{n\to\infty}\frac{\F_c(\phi+v_n)}{\norm{v_n}_{H^1}^2}\geq\mu_c\norm{z}_{H^1}^2+\frac{2\gamma-c}{2}d\mu^*(\R).
\end{equation}
\end{theorem}

\begin{proof}
The normalization gives $\norm{z_n}_{H^1}=1$. For every nonnegative $\psi\in C_c(\R)$, weak lower semicontinuity applied to $\sqrt\psi\,z_{n,x}$ gives
\[\int\psi|z_x|^2\,dx\leq\liminf_n\int\psi|z_{n,x}|^2\,dx=\int\psi\,d\mu,\]
which yields $d\mu\ge |z_x|^2dx$. To show that $d\mu$ is finite, choose $\chi\in C_c^\infty$ even, with $0\leq\chi\leq1$, $\chi=1$ on $[-1,1]$, and decreasing as a function of $|x|$. Set $\chi_R(x)=\chi(x/R)$, then $\chi_R\rightarrow 1$ as $R\to\infty$, and
\[\int\chi_R\,d\mu=\lim_n\int\chi_R|z_{n,x}|^2\,dx\leq1.\]
By the monotone convergence theorem, $d\mu(\R)=\lim_{R\to\infty}\int \chi_R\,d\mu \leq1$.
Consequently $d\mu^*$ is finite, and continuity from below gives the stated formula for $d\mu^*(\R)$.
Local measure convergence need not preserve total derivative mass: mass escaping to infinity
is not included in $d\mu^*$.

The coefficient error in replacing $v_n$ by zero is exactly
\begin{equation}\label{eq:4.13}
B_c(x,v_n,w)-B_c(x,0,w)=-\left(\phi v_n+\frac12v_n^2\right)w^2.
\end{equation}
Thus
\[\left|\frac1{\norm{v_n}_{H^1}^2}\int\bigl(B_c(x,v_n,v_{n,x})-B_c(x,0,v_{n,x})\bigr)\,dx\right|
\leq\left(\norm{\phi}_\infty\norm{v_n}_\infty+\frac12\norm{v_n}_\infty^2\right)\norm{z_{n,x}}_2^2=o(1).\]
By Proposition~\ref{prop:factorization},
\begin{equation}\label{eq:4.14}
\frac{B_c(x,0,v_{n,x})}{\norm{v_n}_{H^1}^2}=\omega_c|z_{n,x}|^2+\frac1{12}|z_{n,x}|^2(v_{n,x}+2\phi_x)^2.
\end{equation}
Since $\chi_R\omega_c$ is continuous and compactly supported, local measure convergence implies
\[\liminf_n\int\omega_c|z_{n,x}|^2\,dx\geq\int\chi_R\omega_c\,d\mu.\]
By the monotone convergence theorem
\begin{equation}\label{eq:4.15}
\liminf_n\int\omega_c|z_{n,x}|^2\,dx\geq\int\omega_c\,d\mu.
\end{equation}
The passage to the limit uses only compactly supported weights and the nonnegativity of the discarded terms.
For the square term, fix $\delta>0$ and set
\[G_n^\delta=\{|v_{n,x}|\leq\delta\},\qquad C_n^\delta=\R\setminus G_n^\delta,\qquad h_n^\delta=z_{n,x}\mathbf1_{G_n^\delta}.\]
Strong convergence in $X$ gives
\begin{equation}\label{eq:4.16}
|C_n^\delta|\leq\delta^{-4}\norm{v_{n,x}}_4^4\longrightarrow0.
\end{equation}
For every fixed $g\in L^2(\R)$, absolute continuity of its $L^2$ integral and boundedness of $z_{n,x}$ imply
\[\left|\int gz_{n,x}\mathbf1_{C_n^\delta}\,dx\right|\leq\norm{z_{n,x}}_2\norm{g\mathbf1_{C_n^\delta}}_2\longrightarrow0.\]
Hence $h_n^\delta\rightharpoonup z_x$ in $L^2$. On $G_n^\delta$, $|v_{n,x}+2\phi_x|\geq(2|\phi_x|-\delta)_+$, where $t_+=\max\{t,0\}$. With $a_\delta=\frac{(2|\phi_x|-\delta)_+^2}{12}$, multiplication by the bounded function $\sqrt{a_\delta}$ preserves weak $L^2$ convergence, and
\[\liminf_n\frac1{12}\int|z_{n,x}|^2(v_{n,x}+2\phi_x)^2\,dx\geq\liminf_n\int a_\delta|h_n^\delta|^2\,dx\geq\int a_\delta|z_x|^2\,dx.\]
By the monotone convergence theorem, one has
\begin{equation}\label{eq:4.17}
\liminf_n\frac1{12}\int|z_{n,x}|^2(v_{n,x}+2\phi_x)^2\,dx\geq\frac13\int\phi_x^2|z_x|^2\,dx.
\end{equation}
 Since $d\mu=|z_x|^2\,dx+d\mu^*$, \eqref{eq:4.15} and \eqref{eq:4.17}, together with \eqref{eq:4.6}, yield \eqref{eq:4.10}.

 It remains to consider the potential and penalty terms. Write $V_c=V_c^+-V_c^-$, where $V_c^\pm=\max\{\pm V_c,0\}$. The limit $V_c(x)\to c-\gamma>0$ implies that $V_c^-$ vanishes outside a bounded interval. Weak $H^1$ convergence implies strong $L^2$ convergence on each bounded interval by compactness (the unique possible local limit is $z$). Therefore
\[\int V_c^-z_n^2\,dx\longrightarrow\int V_c^-z^2\,dx,\qquad\liminf_n\int V_c^+z_n^2\,dx\geq\int V_c^+z^2\,dx,\]
the second inequality following from weak lower semicontinuity of $\norm{\sqrt{V_c^+}\,z_n}_2^2$. It follows that
\begin{equation}\label{eq:4.18}
\liminf_n\int V_cz_n^2\,dx\geq\int V_cz^2\,dx.
\end{equation}
Since $E(u)=\frac12\norm{u}_{H^1}^2$, the exact energy expansion is
\begin{equation}\label{eq:4.19}
E(\phi+v)-E(\phi)=\ip{p_c}{v}+\frac12\norm{v}_{H^1}^2.
\end{equation}
Here $p_c\in L^2$, so
\begin{equation}\label{eq:4.20}
\frac{K_c}{\norm{v_n}_{H^1}^2}\bigl(E(\phi+v_n)-E(\phi)\bigr)^2=K_c\left(\ip{p_c}{z_n}+\frac{\norm{v_n}_{H^1}}{2}\right)^2\longrightarrow K_c\ip{p_c}{z}^2.
\end{equation}
Also \eqref{eq:4.5} gives $\frac{R_c(v_n)}{\norm{v_n}_{H^1}^2}=o(1)$. Since the potential contribution is bounded, the penalty converges, and the derivative contribution is bounded from below by $-o(1)$, the preceding estimates combine to give \eqref{eq:4.11}. Finally, if $\ip{z_n}{\phi_x}=0$, weak convergence and $\phi_x\in L^2$ give $\ip{z}{\phi_x}=0$. Hence \eqref{eq:3.26} and \eqref{eq:4.7} imply \eqref{eq:4.12}.
\end{proof}
Theorem~\ref{thm:concentration} controls both the weak component and the derivative defect at finite spatial locations. It remains to rule out escape of normalized mass to infinity.

\begin{proposition}\label{cor:nonlinear}
There exist $\rho_c>0$ and $C_c>0$ such that
\begin{equation}\label{eq:4.21}
\F_c(\phi+v)\geq C_c\norm{v}_{H^1}^2,
\end{equation}
whenever $v\in X$, $\norm{v}_X<\rho_c$, and $\ip{v}{\phi_x}=0$.
\end{proposition}
\begin{proof}
Suppose the assertion fails. Choosing both the proposed radius and coercivity constant to be $\frac{1}{n}$, we obtain nonzero $v_n\in X$ with
\[\norm{v_n}_X<\frac{1}{n},\qquad\ip{v_n}{\phi_x}=0,\qquad\frac{\F_c(\phi+v_n)}{\norm{v_n}_{H^1}^2}<\frac{1}{n}.\]
Set $z_n=\frac{v_n}{\norm{v_n}_{H^1}}$, then
\begin{equation}\label{eq:4.22}
\norm{z_n}_{H^1}=1,\qquad\limsup_n\frac{\F_c(\phi+v_n)}{\norm{v_n}_{H^1}^2}\leq0.
\end{equation}
By weak compactness in $H^1$ and local weak-* compactness of the derivative measures, after extraction
\[z_n\rightharpoonup z\text{ in }H^1,\qquad |z_{n,x}|^2\,dx\rightharpoonup^*d\mu\text{ locally}.\]
Since $\phi_x\in L^2$, we have \(\langle z,\phi_x\rangle=0\). By Theorem~\ref{thm:concentration},
\[0\geq\limsup_n\frac{\F_c(\phi+v_n)}{\norm{v_n}_{H^1}^2}\geq\liminf_n\frac{\F_c(\phi+v_n)}{\norm{v_n}_{H^1}^2}\geq\mu_c\norm{z}_{H^1}^2+\frac{2\gamma-c}{2}d\mu^*(\R)\geq0.\]
Consequently,
\begin{equation}\label{eq:4.23}
\frac{\F_c(\phi+v_n)}{r_n^2}\longrightarrow0,\qquad z=0,\quad d\mu^*=0,\quad d\mu=0.
\end{equation}
For each fixed $R>0$, take a nonnegative smooth compactly supported cutoff equal to one on $[-R,R]$. Local measure convergence then gives $\int_{-R}^R|z_{n,x}|^2\,dx\to0$. Local compactness gives $\int_{-R}^R|z_n|^2\,dx\to0$. Hence
\begin{equation}\label{eq:4.24}
\int_{-R}^R(|z_n|^2+|z_{n,x}|^2)\,dx\longrightarrow0\quad\text{for every fixed }R>0.
\end{equation}
It remains to exclude escape to infinity. The exact coefficient identity \eqref{eq:4.13} implies, for all $x,w$,
\[B_c(x,v_n,w)\geq\left(\frac{2\gamma-c}{2}-\norm{\phi}_\infty\norm{v_n}_\infty-\frac12\norm{v_n}_\infty^2\right)w^2.\]
Since $\norm{v_n}_\infty\to0$, for all sufficiently large $n$ this gives the global bound
\begin{equation}\label{eq:4.25}
B_c(x,v_n,w)\geq\frac{2\gamma-c}{4}w^2.
\end{equation}
Fix $R$ sufficiently large that
\begin{equation}\label{eq:4.26}
V_c(x)\geq\frac{c-\gamma}{2}\quad(|x|\geq R).
\end{equation}
The radius is now kept fixed while $n\to\infty$. The action expansion, the nonnegative penalty, and \eqref{eq:4.25} yield
\begin{align}\label{eq:4.27}
\frac{\F_c(\phi+v_n)}{\norm{v_n}_{H^1}^2}
&\geq \frac{2\gamma-c}{4}\norm{z_{n,x}}_2^2+\frac{c-\gamma}{4}\int_{|x|>R}|z_n|^2\,dx
-\frac{\norm{V_c^-}_\infty}{2}\int_{-R}^R|z_n|^2\,dx+\frac{R_c(v_n)}{\norm{v_n}_{H^1}^2}\notag\\
&=\frac{2\gamma-c}{4}\norm{z_{n,x}}_2^2+\frac{c-\gamma}{4}\int_{|x|>R}|z_n|^2\,dx+o(1).
\end{align}
The last equality uses \eqref{eq:4.24} and \eqref{eq:4.5}. By \eqref{eq:4.23}, the left-hand side tends to zero. Both terms of \eqref{eq:4.27} are nonnegative with strictly positive coefficients, so
\[\norm{z_{n,x}}_2\longrightarrow0,\qquad\int_{|x|>R}|z_n|^2\,dx\longrightarrow0.\]
Together with local $L^2$ convergence on $[-R,R]$, this yields $\norm{z_n}_{H^1}\to0$, contradicting \eqref{eq:4.22}. The proof is complete.
\end{proof}

\section{Modulation and proof of orbital stability}\label{sec:stability}
With nonlinear Lyapunov coercivity established in Proposition~\ref{cor:nonlinear}, the remaining work is to prove Theorem~\ref{thm:stability}. Quantitative modulation and conservation give $H^1$ control through coercivity, and the Hamiltonian identity then recovers the $W^{1,4}$ component.
Let
\begin{equation}\label{eq:5.1}
\Oc:=\{\phi_c(\cdot-y):y\in\R\},\qquad d_X(u,\Oc):=\inf_{y\in\R}\norm{u-\phi_c(\cdot-y)}_X.
\end{equation}
The following lemma provides the quantitative $X$-modulation estimate needed to apply Proposition~\ref{cor:nonlinear} at a first $X$-exit time.

\begin{lemma}\label{lem:modulation}
There exist $\rho_{\mathrm{mod}}>0$ and $C_{\mathrm{mod}}>0$ such that, whenever $d_X(u,\Oc)<\rho_{\mathrm{mod}}$, there exists $y=y(u)\in\R$ such that
\begin{equation}\label{eq:5.2}
\ip{u(\cdot+y)-\phi}{\phi_x}=0
\end{equation}
and
\begin{equation}\label{eq:5.3}
\norm{u(\cdot+y)-\phi}_X\leq C_{\mathrm{mod}}d_X(u,\Oc).
\end{equation}
\end{lemma}
\begin{proof}
Let $d=d_X(u,\Oc)$. The orbit $\Oc$ is closed in $X$:  a strongly convergent sequence of translates cannot have translation parameters escaping to infinity, while bounded translation parameters admit a convergent subsequence by continuity of translations in $X$. Thus, if $d=0$, then $u\in\Oc$ and the conclusion is immediate. We may therefore assume $d>0$ and choose $y_0\in\R$ so that
\begin{equation}\label{eq:5.4}
\norm{u(\cdot+y_0)-\phi}_X\leq2d.
\end{equation}
Set $h=u(\cdot+y_0)-\phi$. Consider
\begin{equation}\label{eq:5.5}
\Psi(h,\eta)=\ip{h(\cdot+\eta)+\phi(\cdot+\eta)-\phi}{\phi_x}.
\end{equation}
The map $(h,\eta)\mapsto h(\cdot+\eta)$ is $C^1$ from $H^1(\R)\times\R$ to $L^2(\R)$, with derivative in the translation parameter given by $h_x(\cdot+\eta)$. Hence $\Psi$ is $C^1$ near $(0,0)$ as a map $H^1\times\R\to\R$. Moreover, $\Psi(0,0)=0$ and $\partial_\eta\Psi(0,0)=\norm{\phi_x}_2^2>0$. The implicit-function theorem yields, for $\norm{h}_{H^1}$ small, a unique small correction $\eta=\eta(h)$ satisfying $\Psi(h,\eta)=0$ and
\begin{equation}\label{eq:5.6}
|\eta|\leq C\norm{h}_{H^1}\leq2Cd.
\end{equation}
Set $y=y_0+\eta$. Then \eqref{eq:5.2} holds. Since translations are isometries on $X$,
\[\norm{u(\cdot+y)-\phi}_X\leq\norm{h(\cdot+\eta)}_X+\norm{\phi(\cdot+\eta)-\phi}_X=\norm{h}_X+\norm{\phi(\cdot+\eta)-\phi}_X.\]
The smooth exponentially decaying profile satisfies $\phi\in H^2\cap W^{2,4}$,
hence $\norm{\phi(\cdot+\eta)-\phi}_X\leq C|\eta|$ for small $\eta$.
Combining \eqref{eq:5.4} and \eqref{eq:5.6} yields \eqref{eq:5.3}. The proof is complete.
\end{proof}

\begin{proof}[Proof of Theorem~\ref{thm:stability}]
Suppose, to the contrary, that $\phi_c$ is not orbitally stable in $X$. Then there exist initial data $u_{0,n}\in H^s$, $s>\frac{5}{2}$, with
\begin{equation}\label{eq:5.7}
\norm{u_{0,n}-\phi}_X\to0,
\end{equation}
and corresponding strong solutions $u_n(t)$ for which, for some fixed $\varepsilon_0>0$ and times $\tau_n<T_n$,
\begin{equation}\label{eq:5.8}
d_X(u_n(\tau_n),\Oc)\geq\varepsilon_0.
\end{equation}
Choose
\[0<\rho<\min\left\{\frac{\varepsilon_0}{2},\rho_{\mathrm{mod}},\frac{\rho_c}{2C_{\mathrm{mod}}}\right\},\]
where $\rho_c$ is the radius in Proposition~\ref{cor:nonlinear}. By Proposition~\ref{prop:wellposed}, $t\mapsto u_n(t)$ is continuous in $X$. Distance to the fixed closed orbit is $1$-Lipschitz, so $t\mapsto d_X(u_n(t),\Oc)$ is continuous. For all sufficiently large $n$, its initial value is less than $\rho$. Thus the first exit time
\begin{equation}\label{eq:5.9}
t_n^*:=\inf\{t\in[0,\tau_n]:d_X(u_n(t),\Oc)=\rho\}
\end{equation}
is well defined and satisfies
\begin{equation}\label{eq:5.10}
d_X(u_n(t_n^*),\Oc)=\rho.
\end{equation}
By Lemma~\ref{lem:modulation}, there exists a shift $y_n$ such that, with $\widetilde u_n(x)=u_n(t_n^*,x+y_n)$ and $v_n=\widetilde u_n-\phi$, we have
\[\ip{v_n}{\phi_x}=0,\qquad\norm{v_n}_X\leq C_{\mathrm{mod}}\rho<\rho_c.\]
Thus Proposition~\ref{cor:nonlinear} implies that
\begin{equation}\label{eq:5.11}
C_c\norm{v_n}_{H^1}^2\leq\F_c(\widetilde u_n).
\end{equation}
By translation invariance and conservation of $E$ and $H$,
\[\F_c(\widetilde u_n)=\F_c(u_n(t_n^*))=\F_c(u_{0,n}).\]
The map $\F_c:X\to\R$ is continuous, and \eqref{eq:5.7} gives $\F_c(u_{0,n})\to0$. Hence
\begin{equation}\label{eq:5.12}
\widetilde u_n\to\phi\quad\text{strongly in }H^1.
\end{equation}
To recover convergence in $W^{1,4}$, we use the Hamiltonian identity \eqref{eq:2.4}:
\begin{equation}\label{eq:5.13}
\int_\R u_x^4\,dx=12H(u)+3\int_\R u^4\,dx+6\int_\R u^2u_x^2\,dx+6\gamma\int_\R u^2\,dx.
\end{equation}
Conservation of $H$ and \eqref{eq:5.7} imply $H(\widetilde u_n)=H(u_{0,n})\to H(\phi)$. By \eqref{eq:5.12} and the one-dimensional embedding $H^1\hookrightarrow L^\infty$,
\[\widetilde u_n\to\phi\quad\text{in }L^\infty\cap L^2\cap L^4.\]
Furthermore,
\begin{align*}
\left|\int\widetilde u_n^2\widetilde u_{n,x}^2\,dx-\int\phi^2\phi_x^2\,dx\right|
&\leq\norm{\widetilde u_n^2-\phi^2}_\infty\norm{\widetilde u_{n,x}}_2^2
+\norm{\phi}_\infty^2\norm{\widetilde u_{n,x}^2-\phi_x^2}_1\longrightarrow0.
\end{align*}
Thus \eqref{eq:5.13} gives
\begin{equation}\label{eq:5.14}
\norm{\widetilde u_{n,x}}_4^4\to\norm{\phi_x}_4^4.
\end{equation}
The sequence $\widetilde u_{n,x}$ is bounded in $L^4$. Every weak $L^4$ subsequential limit equals $\phi_x$, because \eqref{eq:5.12} already gives convergence of the derivatives in distributions. Hence $\widetilde u_{n,x}\rightharpoonup\phi_x$ in $L^4$. By \eqref{eq:5.14} and uniform convexity of $L^4$, $\widetilde u_{n,x}\to\phi_x$ in $L^4$. Together with \eqref{eq:5.12}, this proves
\begin{equation}\label{eq:5.15}
\widetilde u_n\to\phi\quad\text{strongly in }X.
\end{equation}
Consequently, $d_X(u_n(t_n^*),\Oc)\leq\norm{\widetilde u_n-\phi}_X\to0$, which contradicts \eqref{eq:5.10}.

The equation and the conserved functionals are invariant under $u\mapsto-u$. Therefore the same argument applies to $-\phi_c$.
\end{proof}

\noindent{\bf Acknowledgement} Ji Li was supported by National Key R\&D Program of China (Grant No. 2024YFA1012801)
and National Natural Science Foundation of China (Grant No. 12571174,12625111).
X. Deng was partially supported by National Natural Science Foundation of China (Grant No. 12571172).
Z. Liu was supported by National Natural Science Foundation of China (Grant No. 12571188)
and Natural Science Foundation of Chongqing (No. CSTB2025NSCQ-GPX0705).

\section*{Declaration of competing interest}
The authors declared that they have no conflict of interest.

\section*{Data availability}
No data was used for the research described in the article.

\end{document}